\documentclass[12pt,twoside]{amsart}

\usepackage{amsmath, amssymb, paralist,tabularx,enumitem,supertabular,amsmath, verbatim, amsthm, amssymb, mathrsfs, manfnt}
\usepackage[all]{xy}

\evensidemargin \oddsidemargin
\parskip=\medskipamount
\DeclareMathOperator*{\Gal}{Gal}

\DeclareMathOperator*{\Frob}{Frob_{L/K}}
\DeclareMathOperator*{\FrobclassfieldL}{Frob_{L/K}}

\DeclareMathOperator*{\ordone}{ord_{\nu}}

\DeclareMathOperator*{\gl2}{GL_2}

\DeclareMathOperator*{\m2}{M_2}

\newcommand{\Z}{{\mathbb Z}}

\newcommand{\A}{{\mathfrak A}}

\newcommand{\E}{{\mathcal E}}
\newcommand{\D}{{\mathcal D}}
\newcommand{\Rc}{{\mathcal R}}
\newcommand{\Ok}{{\mathcal{O}_K}}

\newtheorem{thm}{Theorem}[section]
\newtheorem{theorem}[thm]{Theorem}
\newtheorem{corollary}[thm]{Corollary}
\newtheorem{lemma}[thm]{Lemma}	
\newtheorem{proposition}[thm]{Proposition}

\theoremstyle{definition}

\theoremstyle{remark}
\newtheorem{remark}[thm]{Remark}
\newtheorem{example}[thm]{Example}
\newtheorem*{lemma*}{Lemma}

\numberwithin{equation}{section}

\title{Selectivity in Quaternion Algebras}

\author{Benjamin Linowitz}
\address{Department of Mathematics\\ 
6188 Kemeny Hall\\
Dartmouth College\\
Hanover, NH 03755, USA}
\email[] {Benjamin.D.Linowitz@dartmouth.edu}

\begin{document}

\begin{abstract} 
	We prove an integral version of the classical Albert-Brauer-Hasse-Noether theorem regarding quaternion algebras over number fields. Let $K$ be a number field with ring of integers $\mathcal O_K$, and let $\mathfrak A$ be a quaternion algebra over $K$ satisfying the Eichler condition. Let $\Omega$ be a commutative, quadratic $\mathcal{O}_K$-order and let  $\mathcal{R}\subset \mathfrak A$ be an order of full rank. Assume that there exists an embedding of $\Omega$ into $\mathcal R$. We describe a number of criteria which imply that every order in the genus of $\mathcal R$ admits an embedding of $\Omega$. In the case that the relative discriminant ideal of $\Omega$ is coprime to the level of $\mathcal R$ and the level of $\mathcal R$ is coprime to the discriminant of $\A$, we give necessary and sufficient conditions for an order in the genus of $\mathcal R$ to admit an embedding of $\Omega$. We explicitly parameterize the isomorphism classes of orders in the genus of $\mathcal R$ which admit an embedding of $\Omega$. In particular, we show that the proportion of the genus of $\mathcal{R}$ admitting an embedding of $\Omega$ is either $0$, $1/2$ or $1$. Analogous statements are proven for optimal embeddings.
\end{abstract}

\maketitle


\section{Introduction}
 
 The study of non-commutative algebras has a long and rich history with 
applications in class field theory, modular forms and geometry. 
One of the high points of this history came in 1932, when much of the 
field's foundational work was being done, with the publication 
of the Brauer-Hasse-Noether Theorem. Although his name did not appear 
on the publication, many of the results contained in the 
Brauer-Hasse-Noether theorem were proven independently by Adrian Albert 
in 1931. We are interested in the quaternionic version of their theorem:

\begin{theorem}[Albert-Brauer-Hasse-Noether]\label{theorem:abhn}
Let $\A$ be a quaternion algebra over a number field $K$ and let $L$ be a 
quadratic field extension of $K$. Then there is an embedding of $L$ into $\A$ if 
and only if no prime of $K$ which ramifies in $\A$ splits in $L$. 
\end{theorem}

In this paper we consider an integral refinement of the Albert-Brauer-Hasse-Noether theorem. Let $\mathfrak A$ be a quaternion algebra defined over a number field $K$ which satisfies the Eichler condition; that is, there exists an archimedean prime of $K$ which does not ramify in $\mathfrak A$. Let $L$ be a quadratic field extension of $K$ and $\Omega$ a quadratic order of $L$. Finally, let $\mathcal R\subset \A$ be an $\mathcal O_K$-order of full rank. We are interested in the question of when there exists an embedding of $\Omega$ into $\mathcal R$.

Chinburg and Friedman considered a special case of this question in their paper \cite{chinburg}, where they determined the maximal orders into which $\Omega$ can be embedded. We now state a simplified version of their main theorem (see \cite{chinburg} for notation):

\begin{theorem}[Chinburg-Friedman]\label{theorem:introductionchinburgfriedman}
Let $L$ be a quadratic field extension of $K$ which embeds into the quaternion algebra 
$\A$, and assume that $\A$ satisfies the Eichler condition. Then a quadratic $\mathcal{O}_K$-order $\Omega\subset L$ 
can be embedded into either all of the maximal $\Ok$-orders $\D\subset \A$, or into 
all those belonging to exactly half of the isomorphism types of maximal $\Ok$-orders in $\A$. In the latter case the maximal orders admitting 
an embedding of $\Omega$ may be described as follows. If $\mathcal R$ and $\mathcal S$ are maximal orders and $\mathcal R$ admits an embedding of $\Omega$, then
$\mathcal S$ admits an embedding of $\Omega$ if and only if the image by the reciprocity map $\mbox{Frob}_{L/K}$ of the distance ideal $\rho(\mathcal R, \mathcal S)$ is the trivial element of $\mbox{Gal}(L/K)$.
\end{theorem}

Theorem \ref{theorem:introductionchinburgfriedman} is especially significant because of its applications, particularly its applications to differential geometry. As our results will have similar applications, we give a brief sketch of the relevant construction. Suppose that $K\ne \mathbb Q$ is totally real and that $\A$ is unramified at a unique real place of $K$. Let $\mathcal R_1,\mathcal R_2$ be two orders of $\A$ which lie in the same genus (i.e. have locally isomorphic completions at all finite primes of $K$) but represent distinct isomorphism classes. Vigneras \cite{vigneras-isospectral} used $\mathcal R_1$ and $\mathcal R_2$ to construct compact, non-isometric hyperbolic 2-manifolds $M_1$ and $M_2$. Further, she showed that $M_1$ and $M_2$ are isospectral (have the same spectra with respect to the Laplace-Beltrami operator) if there did not exist a quadratic $\mathcal O_K$-order $\Omega$ which could be embedded into exactly one of the $\mathcal R_i$. When $\mathcal R_1$ and $\mathcal R_2$ are taken to be maximal orders, Theorem \ref{theorem:introductionchinburgfriedman} can be used to determine necessary and sufficient conditions for $M_1$ and $M_2$ to be isospectral.

Chinburg and Friedman's theorem was later generalized to Eichler orders independently by Guo and Qin \cite{guo-qin} and Chan and Xu \cite{chan-xu}. It is interesting to note that whereas Guo and Qin make use of tree-theoretic techniques (as Chinburg and Friedman did) in their generalization, Chan and Xu instead use the representation theory of spinor genera.

In this paper we obtain a number of generalizations of Chinburg and Friedman's theorem. Central to our arguments will be the class field $K(\mathcal R)$ associated to the order $\mathcal R$ (defined in Section \ref{section:cftsection} immediately after the proof of Theorem \ref{theorem:bijection}). The class field $K(\mathcal R)$ is an abelian extension of $K$ whose degree is the number of isomorphism classes in the genus of $\mathcal R$, whose Galois group is of exponent $2$ and whose conductor is divisible only by the prime divisors of the level ideal  $N_\mathcal R$ of $\mathcal R$ (defined in Section \ref{section:notation}); that is, the primes $\nu$ of $K$ for which $\mathcal R_\nu$ is not maximal.

In Section \ref{subsection:obstructionstoselectivity} we consider arbitrary orders $\mathcal R\subset \A$ and describe a number of criteria which, if satisfied, imply that every order in the genus of $\mathcal R$ admits an embedding of a given commutative, quadratic $\mathcal O_K$-order $\Omega$. The main result of the section is 

\begin{theorem}\label{theorem:obsructiontheorem} Let $L$ be a maximal subfield of $\A$, $\Omega\subset L$ a quadratic $\mathcal O_K$-order and assume that an embedding of $\Omega$ into $\mathcal R$ exists.
	\begin{enumerate}
		\item If $L\not\subset K(\mathcal{R})$, then every order in the 
		genus of $\mathcal{R}$ admits an embedding of $\Omega$.
		
		\item If $L\subset K(\mathcal{R})$, then the proportion of isomorphism classes in the genus of $\mathcal R$ whose representatives admit an embedding of $\Omega$ is at least $1/2$.
	\end{enumerate}

	\end{theorem}
	
An easy consequence of Theorem \ref{theorem:obsructiontheorem} is that if $\Omega$ embeds into an order $\mathcal R$, then every representative of at least half of the isomorphism classes in the genus of $\mathcal R$ admits an embedding of $\Omega$. Another consequence is that if any prime of $K$ which does not divide the level ideal $N_\mathcal R$ ramifies in $L$ then every order in the genus of $\mathcal R$ admits an embedding of $\Omega$. This is especially nice in applications because in practice the computation of $K(\mathcal R)$ can be quite difficult. In a future publication we will apply Theorem \ref{theorem:obsructiontheorem} in order to construct \textit{isospectral towers} of hyperbolic manifolds; that is, pairs of infinite towers $\{M_i\},\{N_i\}$ of finite covers of hyperbolic manifolds $M$ and $N$ such that the covers $M_j$ and $N_j$ are isospectral, but not isometric for every $j$.

In Section \ref{subsection:aselectivitytheorem} we constrain slightly the class of orders $\mathcal R \subset \A$ that we consider and provide necessary and sufficient conditions for an order $\Omega$ to embed into some, but not all, orders in the genus of $\mathcal R$. The main result of Section \ref{subsection:aselectivitytheorem} is:

\begin{theorem}\label{theorem:introductionmain}
Assume that an embedding of $\Omega$ into $\mathcal{R}$ exists. Assume as well that the relative discriminant ideal $d_{\Omega/\mathcal{O}_K}$ of $\Omega$ is coprime to the level ideal $N_\mathcal R$ of $\mathcal R$ and that 
the set of primes dividing $N_\mathcal R$ is disjoint from the set of primes ramifying in $\A$. Then every order in the genus of $\mathcal{R}$ admits an embedding of $\Omega$ except when the following conditions hold:

\begin{enumerate}

\item $\Omega$ is an integral domain whose quotient field $L$ is a quadratic field extension of $K$ which is contained in $\A$.

\item There is a containment of fields $L\subset K(\mathcal{R})$.

\item All primes of $K$ which divide the relative discriminant ideal $d_{\Omega/\mathcal{O}_K}$ of $\Omega$ split in $L/K$.

\end{enumerate}

Suppose now that (1) - (3) hold. Then the isomorphism classes in the genus of $\mathcal{R}$ whose representatives admit an embedding of $\Omega$ comprise exactly half of the isomorphism classes. If $\mathcal{R}$ admits an embedding of $\Omega$ and $\mathcal E$ is another order in the genus of $\mathcal R$, then $\mathcal E$ admits an embedding of $\Omega$ if and only if $\displaystyle\FrobclassfieldL(\rho(\mathcal R,\mathcal E))$ is the trivial element of $Gal(L/K)$.
\end{theorem}

The proof of Theorem \ref{theorem:introductionmain} makes extensive use of the tree of maximal orders of $M_2(k)$ (for $k$ a local field), allowing us to parameterize explicitly the orders in the genus of $\mathcal R$ admitting an embedding of $\Omega$, explicit enough to specify them via the local-global correspondence.

In Section \ref{section:optimal} we consider the related question of when there exists an \textit{optimal embedding} of $\Omega$ into $\mathcal R$. Maclachlan \cite{mac} considered Eichler orders of square-free level and proved that the proportion of isomorphism classes of orders in the genus of $\mathcal R$ admitting an optimal embedding of $\Omega$ is equal to $0,\frac{1}{2}$ or $1$. We show that Theorems \ref{theorem:obsructiontheorem} and \ref{theorem:introductionmain} hold not only for embeddings but for optimal embeddings as well. These theorems are of independent interest however, in part because of the ubiquity of optimal embeddings in number theory. For example, optimal embeddings play an important role in Hijikata's \cite{hijikata} formula for the trace of Hecke operators acting on spaces of $\mathcal R^\times$-automorphic cusp forms.

It is a pleasure to thank Tom Shemanske, my advisor, for his encouragement and
for patiently reading and commenting on drafts of this paper. I would also like to thank
Ted Chinburg for useful conversations regarding the proof of Theorem 3.3 of \cite{chinburg}.


\section{Notation}\label{section:notation}
In this section we fix the notation concerning quaternion algebras and their orders 
that will be used throughout this paper.

Let $K$ be a number field with ring of integers $\mathcal{O}_K$. Let $\A $ be a 
quaternion algebra over $K$ with reduced norm $n$. We denote by 
$K_\nu$ the completion of $K$ at a prime $\nu$ of $K$. If $\nu$ 
is a non-archimedean prime, we let $\mathcal{O}_{K_\nu}$ be the valuation 
ring of $K_\nu$ and $\pi_\nu$ a fixed uniformizer. When there will be no confusion we will 
write $\mathcal{O}_\nu$ in place of $\mathcal{O}_{K_\nu}$. We denote by $\A_\nu$ 
the $K_\nu$-algebra $\A\otimes_K K_\nu$ and by $\A^1$ (resp. $\A_\nu^1$) the elements of $\A$ (resp. $\A_\nu$) having reduced norm equal to $1$.  We let $J_K$ denote the idele group of
$K$ and $J_\A$ the idele group of $\A$.

Given an $\mathcal O_K$-order $\mathcal{R}\subset\A$ (having maximal rank) and a prime $\nu$ of $K$, we define the 
completions $\mathcal{R}_\nu\subset \A_\nu$ by:

\begin{displaymath}
\mathcal{R}_\nu = \left\{ \begin{array}{ll}
\mathcal{R}\otimes_{\mathcal{O}_K} \mathcal{O}_{\nu} & \textrm{if $\nu$ is non-archimedean,}\\
\mathcal{R}\otimes_{\mathcal{O}_K} K_\nu=\A_\nu & \textrm{if $\nu$ is archimedean.}\\
\end{array} \right.
\end{displaymath}

For almost all finite primes $\nu$ of $K$, $\A_\nu\cong \displaystyle\m2(K_\nu)$ and 
$\mathcal{R}_\nu\cong \displaystyle\m2(\mathcal{O}_\nu)$, so we will identify $\mathfrak A_\nu$ with $\displaystyle\m2(K_\nu)$. Define the normalizer of $\mathcal R_\nu$ in $\A_\nu^\times$ to be $N(\mathcal R_\nu)=\{ x\in\mathfrak A_\nu^\times : x \mathcal R_\nu x^{-1}=\mathcal R_\nu \}$. This normalizer, along with its image under the reduced norm, is central to determining whether the class field associated to $\mathcal R$ admits finite ramification. We therefore note that whenever $\nu$ is a finite prime of $K$ which is unramified in $\mathfrak A$ and $\mathcal R_\nu$ is a maximal order of $\mathfrak A_\nu$, its normalizer is conjugate to 
$\displaystyle\gl2(\mathcal{O}_\nu)K_\nu^\times$ and hence,
$n(N(\mathcal{R}_\nu))=\mathcal{O}_\nu^\times {K_\nu^\times}^2$.

We define the \textit{level ideal} $N_\mathcal R$ of $\mathcal R$ to be the order-ideal \cite[p. 49]{reiner} of the $\mathcal O_K$-module $\mathcal M/\mathcal R$, where $\mathcal M$ is any maximal order of $\A$ containing $\mathcal R$. This definition is independent of the choice of maximal order $\mathcal M$.


\section{The class field associated to an order}\label{section:cftsection}
Let $\A$ be a quaternion algebra over a number field $K$ satisfying the Eichler condition; that is, there 
exists an archimedean prime of $K$ which splits in $\A$. The reason for this assumption, as will soon be made clear, is that it is only in this context that the Strong Approximation theorem may be applied.

We say that two orders $\mathcal{R}_1$ and $\mathcal{R}_2$ are of the same \emph{genus} if ${\mathcal{R}_1}_\nu$ is 
$\mathcal{O}_\nu$-isomorphic to ${\mathcal{R}_2}_\nu$ for all finite primes $\nu$ of $K$. The \textit{type number} $t(\mathcal{R})$ of an order 
$\mathcal{R}$ is defined to be the number of isomorphism classes of orders in the genus of $\mathcal{R}$.

\begin{example}
Let $\mathcal M\subset \mathfrak A$ be a maximal order (of full rank). For every finite prime $\nu$ of $K$, $\mathcal M_\nu$ is a maximal order of $\mathfrak A_\nu$. If $\nu$ is a finite prime splitting in $\mathfrak A$, then it is well known that every maximal order of $\mathfrak A_\nu$ is conjugate to $\mathcal M_\nu$. If $\nu$ is a finite prime ramifying in $\mathfrak A$, then $\mathcal M_\nu$ is the unique maximal order of $\mathfrak A_\nu$. It follows that the collection of all maximal orders of $\mathfrak A$ comprise the genus of $\mathcal M$.
\end{example}

\begin{example}
Let $\mathcal E \subset \mathfrak A$ be an Eichler order of level $\mathfrak N$ (where $\mathfrak N$ is coprime to the discriminant of $\mathfrak A$). Recall that this means that $\mathcal E$ is an order of $\mathfrak A$ (of full rank) with the property that $\mathcal E_\nu$ is maximal if $\nu$ ramifies in $\mathfrak A$, and if $\nu$ splits in $\A$ then $\mathcal E_\nu$ is conjugate to \begin{displaymath}
	\begin{pmatrix}
	  \mathcal O_\nu &   \mathcal O_\nu\\
	 \mathfrak N  \mathcal O_\nu &    \mathcal O_\nu
	\end{pmatrix}.
\end{displaymath} It follows that the genus of $\mathcal E$ is the collection of all Eichler orders of level $\mathfrak N$.
\end{example}

Let $\mathcal{R}\subset\A$ be an order (of maximal rank). The isomorphism classes of orders in the 
genus of $\mathcal{R}$ are in one-to-one correspondence with points in the double coset space $\A^\times \backslash J_\A / \mathfrak{N}(\mathcal{R})$, where 
$\mathfrak{N}(\mathcal{R})=J_\A \cap \prod_{\nu}N(\mathcal{R}_\nu)$. This bijection is induced by the map sending an order $\mathcal E$ belonging to the genus of $\mathcal R$ to the double coset $\A^\times \tilde{x}_\mathcal E \mathfrak N(\mathcal R)$, where $\tilde{x}_\mathcal E=(x_{\mathcal E_\nu})$ is an element of $J_\A$ such that $x_{\mathcal E_\nu}\mathcal E_\nu x_{\mathcal E_\nu}^{-1}=\mathcal R_\nu$ for all $\nu$.

\begin{theorem}\label{theorem:bijection}
The reduced norm induces a bijection of sets 
\begin{displaymath}
n: \A^\times \backslash J_\A/\mathfrak{N}(\mathcal{R})\longrightarrow K^\times \backslash J_K / n(\mathfrak{N}(\mathcal{R})).
\end{displaymath}
\end{theorem}

\begin{proof} 
	
	The induced map is defined in the obvious way: $n(\mathfrak A^\times \tilde{\alpha} \mathfrak N(\mathcal R))=K^\times n(\tilde{\alpha}) n(\mathfrak N(\mathcal R))$.

We first show that $n$ is surjective. Let $\tilde{\beta}\in J_K$. We may assume, by the weak approximation theorem, that $\beta_\nu>0$ for every infinite prime $\nu$ ramifying in $\mathfrak A$. The reduced norm is locally surjective at all finite primes and at the infinite primes not ramifying in $\mathfrak A$. At the infinite primes ramifying in $\mathfrak A$ the image of the reduced norm consists of the non-negative reals. We construct an idele $\tilde{\alpha}\in J_\mathfrak A$ such that $n(\mathfrak A^\times \tilde{\alpha} \mathfrak N(\mathcal R))=K^\times \tilde{\beta} n(\mathfrak N(\mathcal R))$. For all but finitely many non-archimedean primes $\nu$ of $K$, $\beta_\nu\in \mathcal{O}_\nu^\times$ and $\mathcal R_\nu\cong M_2(\mathcal{O}_\nu)$. For each such prime $\nu$, let $\alpha_\nu$ be the conjugate of the diagonal matrix $\mbox{diag}(\beta_\nu,1)$ lying in $\mathcal{R}_\nu$. For the other primes, define $\alpha_\nu$ to be any element in the preimage of $\beta_\nu$. The constructed element $\tilde{\alpha}=(\alpha_\nu)_\nu$ is easily seen to lie in $J_\mathfrak A$ and establishes surjectivity.

We now show that $n$ is injective. Denote by $J_\mathfrak A^1$ the kernel of the reduced norm map $n: J_\mathfrak A\rightarrow J_K$. We show that $\mathfrak A^\times J_\mathfrak A^1 \mathfrak {N}(\mathcal R)$ is the preimage of $K^\times n(\mathfrak N(\mathcal R))$. It is obvious that $\mathfrak A^\times J_\mathfrak A^1 \mathfrak {N}(\mathcal R)$ is contained in the preimage. Let $\tilde{\gamma}\in J_\mathfrak A$ be such that $n(\mathfrak A^\times \tilde{\gamma} \mathfrak N(\mathcal R))=K^\times n(\mathfrak N(\mathcal R))$. Then $n(\tilde{\gamma})=k\cdot n(\tilde{r})$ where $k\in K^\times$ and $\tilde{r}\in \mathfrak N (\mathcal R)$. If $\nu$ is an infinite prime ramifying in $\mathfrak A$ then $n(\gamma_\nu), n(r_\nu)>0$. Thus $k_\nu>0$ as well. The Hasse-Schilling-Maass theorem (Theorem 33.15 of \cite{reiner}) implies that there exists $b\in\mathfrak A^\times$ such that $n(b)=k$ and $n(\tilde{\gamma})=n(b)n(\tilde{r})$, hence $n(b^{-1})n(\tilde{\gamma})n(\tilde{r}^{-1})=(1)\in J_K$. This shows that $\mathfrak A^\times \tilde{\gamma}\mathfrak N(\mathcal R)=\mathfrak A^\times b^{-1}\tilde{\gamma}\tilde{r}^{-1}\mathfrak{N}(\mathcal R)\in \A^\times J_\mathfrak A^1\mathfrak N(\mathcal R)$ as claimed.

We now continue with our proof of injectivity. Suppose that there exist $\tilde{\alpha},\tilde{\beta}\in J_\mathfrak A$ such that $n(\mathfrak A^\times\tilde{\alpha}\mathfrak N{(\mathcal R)})=n(\mathfrak A^\times\tilde{\beta}\mathfrak N{(\mathcal R)})$. Then $n(\tilde{\alpha}^{-1}\tilde{\beta})\in K^\times n(\mathfrak N(\mathcal R))$ and by the above claim $\tilde{\alpha}^{-1}\tilde{\beta}\in\mathfrak A^\times J_\mathfrak A^1 \mathfrak N(\mathcal R)$. Making use of the fact that $\mathfrak A^\times J_\mathfrak A^1$ is a normal subgroup of $J_\mathfrak A$, we see that $\tilde{\beta}\in\tilde{\alpha}\mathfrak A^\times J_\mathfrak A^1\mathfrak N(\mathcal R)=\mathfrak A^\times J_\mathfrak A^1\tilde{\alpha}\mathfrak N(\mathcal R)$. 

Let $S$ be the set of archimedean places of $K$ and define $\A_S^1=\prod_{\nu\in S} \A_\nu^1$. Recall the Strong Approximation theorem (\cite[Theorem 4.3]{vigneras}): If $\A^1_S$ is not compact, then $\A^1 \A^1_S$ is dense in $J_\A^1$. Recalling that $\A^1_\nu$ is compact if and only if $\nu$ ramifies in $\A$, our assumption that $\A$ satisfy the Eichler condition implies that there exists an archimedean prime $\nu$ of $K$ such that $\A^1_\nu$ is not compact. It follows that $\A^1_S$ is not compact. For any $\tilde{\gamma}\in J_\A$, $\tilde{\gamma}\mathfrak N(\mathcal R)\tilde{\gamma}^{-1}$ contains a neighborhood of the identity and is therefore an open subgroup of $J_\A$ containing $\A^1_S$, hence $J_\A^1\subset \A^\times \tilde{\gamma}\mathfrak N(\mathcal R)\tilde{\gamma}^{-1}$.
 Choosing $\tilde{\gamma}=\tilde{\alpha}$, we have $$\tilde{\beta}\in \mathfrak A^\times J_\mathfrak A^1 \tilde{\alpha}\mathfrak N(\mathcal R)\subset \mathfrak A^\times \tilde{\alpha} \mathfrak N(\mathcal R).$$
Therefore $\mathfrak A^\times \tilde{\beta} \mathfrak N(\mathcal R)\subset \mathfrak A^\times \tilde{\alpha} \mathfrak N(\mathcal R)$, and by symmetry, we have equality.\end{proof}

We have shown that the isomorphism classes comprising the genus of a fixed 
order $\mathcal{R}\subset\A$ are in one-to-one correspondence
with the double cosets of the group $K^\times \backslash J_K / n(\mathfrak{N}(\mathcal{R}))$. 
Set $H_\mathcal{R}=K^\times n(\mathfrak{N}(\mathcal{R}))$ 
and $G_\mathcal{R}=J_K/H_\mathcal{R}$. As $J_K$ is abelian, $G_\mathcal{R} \cong K^\times \backslash J_K / n(\mathfrak{N}(\mathcal{R}))$. Since $H_\mathcal{R}$ is an open subgroup of $J_K$, associated to it is a class field $K(\mathcal{R})$ 
whose arithmetic is intimately related to the arithmetic of $\mathcal{R}$ in $\A$. The basic properties of $K(\mathcal{R})$ are given by the standard theorems of class field theory (see for example Chapter 11 of \cite{lang}). In particular, we note that the conductor of $K(\mathcal R)$ is divisible only by prime divisors of the level ideal $N_\mathcal R$ of $\mathcal R$.

We now prove that $G_\mathcal{R}$ can be generated by elements
having a very simple form and that in fact, $G_\mathcal R$ is an elementary abelian 
group of exponent $2$.

\begin{lemma}\label{lemma:ggen}
$G_\mathcal{R}$ is generated by cosets having representatives of the form \\ $e_{\nu_i}=(1,...,1,\pi_{\nu_i},1,...)$. If $S$ is any finite set of primes of $K$, then the representatives $\{e_{\nu_i}\}$ can be taken so that $\nu_i\notin S$ 
for all $i$.
\end{lemma}

\begin{proof}
The Chebotarev Density Theorem guarantees that every element of 
$\displaystyle\Gal(K(\mathcal{R})/K)$ has infinitely many prime ideals in 
its preimage under the Artin map. As these prime ideals
correspond to ideles of the form $(1,1,...,1,\pi_{\nu_i},1,...)$, 
$G_\mathcal R$ can be generated by cosets having representatives of the form $e_{\nu_i}$. 
Only finitely many primes of $K$ 
lie in $S$ and a given element of $G_\mathcal{R}$ has infinitely many prime ideals in 
its preimage, so each $\nu_i$ can be chosen so that $\nu_i\notin S$.
\end{proof}

\begin{proposition}\label{prop:twogrp}
The group $G_\mathcal R$ is an 
elementary abelian group of exponent $2$.
\end{proposition}

\begin{proof}
This is clear since $J_K^2\subset n(\mathfrak N (\mathcal R))$.
\end{proof}

Given orders $\mathcal D, \mathcal E$ in the genus of $\mathcal R$, we define the \textit{distance} $\rho(\mathcal D,\mathcal E)$ as follows. Let $\tilde{x}_\mathcal D, \tilde{x}_\mathcal E$ be defined as in the paragraph preceding Theorem \ref{theorem:bijection}. We define $\rho (\mathcal D, \mathcal E )$ to be the coset $n(\tilde{x}_\mathcal D^{-1}  \tilde{x}_\mathcal E ) H_\mathcal R$ in $G_\mathcal R$. By Proposition \ref{prop:twogrp} this is the same coset as $n(\tilde{x}_\mathcal D  \tilde{x}_\mathcal E ) H_\mathcal R$. It is not difficult to see that our definition of $\rho(-,-)$ is well-defined. Indeed, let $\tilde{x}^\prime_\mathcal D, \tilde{x}^\prime_\mathcal E$ be another pair of ideles such that $x^\prime_{\mathcal D_\nu}\mathcal D_\nu ({x^\prime_{\mathcal D_\nu}})^{-1}=\mathcal R_\nu=x^\prime_{\mathcal E_\nu}\mathcal E_\nu ({x^\prime_{\mathcal E_\nu}})^{-1}$ for all $\nu$. It is clear from Theorem \ref{theorem:bijection} and the paragraph preceding it that this implies $\A^\times \tilde{x}_\mathcal D \mathfrak N(\mathcal R)=\A^\times \tilde{x}^\prime_\mathcal D \mathfrak N(\mathcal R)$ and $\A^\times \tilde{x}_\mathcal E \mathfrak N(\mathcal R)=\A^\times \tilde{x}^\prime_\mathcal E \mathfrak N(\mathcal R)$, hence $n(\tilde{x}_\mathcal D  \tilde{x}_\mathcal E ) H_\mathcal R= n(\tilde{x}^\prime_\mathcal D  \tilde{x}^\prime_\mathcal E ) H_\mathcal R$.
Similar arguments show that the following elementary properties are satisfied.

\begin{proposition}\label{proposition:distanceproperties}
Let $\mathcal D,\mathcal E, \mathcal E^\prime$ lie in the genus of $\mathcal R$.

\begin{enumerate}
	\item $\rho(\mathcal D,\mathcal E)=\rho(\mathcal E,\mathcal D)$
	\item $\rho(\mathcal D,\mathcal E)$ is trivial if and only if $\mathcal D\cong \mathcal E$
	\item If $\mathcal E\cong \mathcal E^\prime$ then $\rho(\mathcal D,\mathcal E)=\rho(\mathcal D,\mathcal E^\prime)$
\end{enumerate}
\end{proposition}

Let $L$ be a quadratic field extension of $K$. We have shown that $G_\mathcal{R}$ is an elementary abelian group of exponent $2$ with generators $\overline{e_{\nu_i}}=e_{\nu_i}H_\mathcal R$. We now show
that the generators $\{\overline{e_{\nu_i}}\}$ can be chosen so that the primes $\nu_i$ have certain splitting properties 
in $L/K$.

\begin{lemma}\label{lemma:splitornot}
Let notation be as above.
\begin{compactenum}
\item If $L\not\subset K(\mathcal{R})$, then $G_\mathcal{R}$ is generated by elements $\{\overline{e}_{\nu_i}\}$ where 
$\nu_i$ splits in $L$ for all $i$.

\item If $L\subset K(\mathcal{R})$, then $G_\mathcal{R}$ is generated by elements $\{\overline{e}_{\nu_i}\}$ where 
$\nu_i$ splits in $L$ for all $i>1$, and $\nu_1$ is inert in $L$.
\end{compactenum}
\end{lemma}
\begin{proof}

We first suppose that $L\not\subset K(\mathcal R)$. By the Chebotarev density theorem we may generate $\mbox{Gal}(K(\mathcal R)L/L)$ with the Frobenius elements associated to primes of $L$ having degree one over $K$ (since the set of primes of $L$ with degree greater than one over $K$ has density zero). As $\mbox{Gal}(K(\mathcal R)L/L)$ is isomorphic to $\mbox{Gal}(K(\mathcal R)/K)$ via restriction to $K(\mathcal R)$, we may generate the latter group with Frobenius elements associated to primes of $K$ splitting completely in $L/K$. These automorphisms correspond, via the Artin map, to the generators $\{\overline{e_{\nu_j}}\}_{j=1}^m$
of $G_\mathcal R$. We have therefore proven the first assertion. 

Suppose now that $L\subset K(\mathcal R)$. Let $\lambda$ be a prime of $K$ which is inert in the extension $L/K$ and set $\nu_1=\lambda$. Indeed, viewing $\displaystyle\Gal(K(\mathcal{R})/L)$
as a subgroup of $\displaystyle\Gal(K(\mathcal{R})/K)$, consider the set $\{(\lambda, K(\mathcal{R})/K), \{(\nu_j,K(\mathcal{R})/K)\}_{j=2}^m \}$
where $\{(\nu_j,K(\mathcal{R})/K)\}_{j=2}^m$ is an $m-1$ element generating set of $\displaystyle\Gal(K(\mathcal{R})/L)$. We claim that this set
generates $\displaystyle\Gal(K(\mathcal{R})/K)$. As $\displaystyle\Gal(K(\mathcal{R})/K)$ and $\displaystyle\Gal(K(\mathcal{R})/L)$ are elementary abelian
groups of exponent $2$, they are $\mathbb{F}_2$-vector spaces. The set $\{(\nu_j,K(\mathcal{R})/K)\}_{j=2}^m$ is therefore a basis of $\displaystyle\Gal(K(\mathcal{R})/L)$,
hence linearly independent. To show that $\{(\lambda, K(\mathcal{R})/K), \{(\nu_j,K(\mathcal{R})/K)\}_{j=2}^m \}$ is a basis for $\displaystyle\Gal(K(\mathcal{R})/K)$,
it suffices to show that $(\lambda, K(\mathcal{R})/K)$ is not an $\mathbb{F}_2$-linear combination of elements of $\{(\nu_j,K(\mathcal{R})/K)\}_{j=2}^m$. But this is clear
as all of the elements of $\{(\nu_j,K(\mathcal{R})/K)\}_{j=2}^m$ restrict to the trivial element of $\displaystyle\Gal(L/K)$ while $(\lambda, K(\mathcal{R})/K)$ does not.\end{proof}



\section{Parameterizing the genus of $\mathcal R$}\label{section:parameterizing}

Let $\{e_{\nu_i}\}_{i=1}^m\subset J_K$ be such that $\{\overline{e_{\nu_i}}\}_{i=1}^m$ 
generate $G_\mathcal{R}=J_K /K^\times n(\mathfrak{N}((\mathcal{R}))\cong (\Z/2\Z)^m$. By
Lemma \ref{lemma:ggen} one may choose this generating set so that each $\nu_i$ 
is non-archimedean and split in $\A$ and so that $\mathcal{R}_{\nu_i}$ is maximal for $1\leq i \leq m$. Throughout 
the remainder of this paper we will only consider generating sets $\{\overline{e_{
\nu_i}}\}_{i=1}^m$ of $G_\mathcal{R}$ which satisfy these properties. For each $1\leq i
\leq m$, let $\mathcal{R}_{\nu_i}^\prime$ be a maximal order of $\A_{\nu_i}$ which is 
adjacent to $\mathcal{R}_{\nu_i}$ in the tree of maximal orders of $\A_{\nu_i}\cong 
M_2(K_{\nu_i})$. For the basic definitions and properties concerning the tree of maximal
orders of $M_2(k)$ (for $k$ a local field) we refer the reader to Section 2.2 of \cite{vigneras}.

Given $\gamma=(\gamma_i)\in (\Z/2\Z)^m$, we define the order $D^\gamma$ (via the local-global correspondence) as having the following local factors

\begin{displaymath}
\mathcal{D}_\nu^\gamma = \left\{ \begin{array}{ll}
\mathcal{R}_{\nu_i} & \textrm{if $\nu=\nu_i$ and $\gamma_i=0$}\\
\mathcal{R}_{\nu_i}^\prime & \textrm{if $\nu=\nu_i$ and $\gamma_i=1$}\\
\mathcal{R}_\nu & \textrm{otherwise.}
\end{array}\right.
\end{displaymath}

We will show that the set $\{\mathcal{D}^\gamma\}_{\gamma\in (\Z/2\Z)^m}$ consists of representatives of all $2^m$ isomorphism classes
in the genus of $\mathcal{R}$. We therefore call $\{\mathcal{D}^\gamma\}_{\gamma\in (\Z/2\Z)^m}$ a \textit{parameterization} of the genus of $\mathcal{R}$. Note that by 
construction, if $\gamma=(0,0,...)$, $\mathcal{D}^\gamma$ is $\mathcal{R}$.

\begin{proposition}
Let notation be as above. The orders of $\{\mathcal{D}^\gamma\}_{\gamma\in (\Z/2\Z)^m}$ are pairwise non-isomorphic and represent all isomorphism classes of the genus of $\mathcal{R}$. In particular, $\mathcal D^\gamma\cong\mathcal D^{\gamma^\prime}$ if and only if $\gamma=\gamma^\prime$.
\end{proposition}

\begin{proof} 
It follows from \cite[Section 2.2]{vigneras} that if $1\leq i\leq m$ and $x\in \A_{\nu_i}^\times$ is such that $x\mathcal R^\prime_{\nu_i}x^{-1}=\mathcal R_{\nu_i}$ then $ord_{\nu_i}(n(x))\equiv 1\pmod{2}$. Using this, it is easy to see that $\rho(\mathcal{D}^\gamma,\mathcal{D}^{\gamma^\prime})=\prod_{i=1}^m \overline{e}_{\nu_i}^{\gamma_i + \gamma_i^\prime \pmod{2}}$. By Proposition \ref{proposition:distanceproperties}, $\mathcal{D}^\gamma\cong \mathcal{D}^{\gamma^\prime}$ if and only if $\rho(\mathcal D^\gamma, \mathcal D^{\gamma^\prime})$ is trivial. As there are $2^m$ non-isomorphic orders $\mathcal D^\gamma$ and $2^m$ isomorphism classes in the genus of $\mathcal R$, the proposition follows.\end{proof}


\section{Selectivity}\label{section:main}

Let $\Omega$ be a commutative, quadratic $\mathcal O_K$-order and assume that an embedding of $\Omega$ into $\mathcal R$ exists. We are interested in determining the proportion of isomorphism classes in the genus of $\mathcal R$ whose representatives admit an embedding of $\Omega$. As in \cite{chinburg}, we shall say that $\Omega$ is \textit{selective} for the genus of $\mathcal{R}$ if $\Omega$ does not embed into every order in the genus of $\mathcal{R}$.

\begin{remark}
Suppose that $\mathcal E$ is an order in the genus of $\mathcal R$,  $\Omega$ is an integral domain and $\varphi$ is an embedding of $\Omega$ into $\mathcal E$. Then $\varphi$ extends to an embedding of the quotient field of $\Omega$ into $\mathfrak A$. Such an embedding is, by the Skolem-Noether theorem, given by conjugation by an element of $\mathfrak A^\times$. Thus an order $\mathcal E$ in the genus of $\mathcal R$ admits an embedding of $\Omega$ if and only if it contains a conjugate of $\Omega$.
\end{remark}

\subsection{Obstructions to selectivity}\label{subsection:obstructionstoselectivity}

In this section we describe several easy to check criteria that, if satisfied, preclude the possibility of selectivity. In order to show that in these cases $\Omega$ can be embedded into every order in the genus of $\mathcal R$, we employ the following strategy.

Let $\{e_{\nu_i}\}_{i=1}^m\subset J_K$ be such that $\{\overline{e_{\nu_i}}\}$ generate 
$G_\mathcal{R}=J_K /{K^\times \cdot n(\mathfrak{N}(\mathcal{R}))}$. 
Note that by Lemma \ref{lemma:ggen} we may assume that each $\nu_i$ 
is non-archimedean, that no $\nu_i$ ramifies in $\A$ and that $\mathcal{R}_{\nu_i}$ is 
maximal for all $i$. For each $\nu_i$, we will construct two distinct local maximal orders 
$\D_{\nu_i}$ and $\D_{\nu_i}^\prime$ which both contain $\Omega$ (that is, $\Omega\subset\D_{\nu_i}\cap\D_{\nu_i}^\prime$ for all $\nu_i$) 
and are adjacent in the tree of maximal orders of $\A_{\nu_i}\cong M_2(K_{\nu_i})$. 
By hypothesis $\mathcal R$ also contains $\Omega$. Then for each $\gamma\in (\mathbb{Z}/2\mathbb{Z})^m$, 
let $\mathcal D^\gamma$ be the global order defined by the local factors

\begin{displaymath}
\mathcal{D}_\nu^\gamma = \left\{ \begin{array}{ll}
\mathcal{D}_{\nu_i} & \textrm{if $\nu=\nu_i$ and $\gamma_i=0$}\\
\mathcal{D}_{\nu_i}^\prime & \textrm{if $\nu=\nu_i$ and $\gamma_i=1$}\\
\mathcal{R}_\nu & \textrm{otherwise.}
\end{array}\right.
\end{displaymath}

As $\Omega$ is contained in every completion of $\mathcal{D}^\gamma$ (for each $\gamma\in (\mathbb{Z}/2\mathbb{Z})^m$), $\Omega$ is contained in $\mathcal{D}^\gamma$ as well. 
These $2^m$ orders represent all isomorphism classes of orders in the genus of $\mathcal{R}$ and each contains a conjugate of $\Omega$. Hence $\Omega$ can be embedded into all orders 
in the genus of $\mathcal{R}$.

\begin{remark}The strategy outlined above was used in \cite{chinburg} by Chinburg and Friedman in order to prove the theorem in the case that $\mathcal{R}$ was a maximal order. 
The same strategy was used by Guo and Qin in \cite{guo-qin} to extend the theorem to Eichler orders of arbitrary level.\end{remark}

\begin{proposition}\label{proposition:case1}
If  $\Omega$ is not an integral domain then every order in the genus of $\mathcal{R}$ admits an embedding of $\Omega$.
\end{proposition}
\begin{proof}

If $\Omega$ is not an integral domain then $\A\cong M_2(K)$ and $\Omega$ is conjugate to a subring of 

\begin{equation*}
\Omega_0=\{ \begin{pmatrix}
  a & 0\\
  0 &  b
\end{pmatrix} : a,b\in\mathcal{O}_{K}\}
\qquad\mbox{ or }\qquad  \Omega_J=\{\begin{pmatrix}
  a & b\\
  0 &  a
\end{pmatrix} : a\in\mathcal{O}_{K}, b\in J\}\end{equation*}

for $J$ a fractional ideal of $K$. 

We may, without loss of generality, assume that $\Omega\subset\Omega_0$ or $\Omega\subset\Omega_J$.

If $\Omega\subset\Omega_0$, set

\begin{displaymath}
\mathcal{D}_{\nu_i}=\begin{pmatrix}
  \mathcal{O}_{\nu_i} & \mathcal{O}_{\nu_i}\\
\mathcal{O}_{\nu_i} &  \mathcal{O}_{\nu_i}
\end{pmatrix},
\qquad
\mathcal{D}_{\nu_i}^\prime=
\begin{pmatrix}
  \mathcal{O}_{\nu_i} & \pi_{\nu_i}^{-1}\mathcal{O}_{\nu_i}\\
  \pi_{\nu_i}\mathcal{O}_{\nu_i} &  \mathcal{O}_{\nu_i}
\end{pmatrix}.
\end{displaymath}

If $\Omega\subset\Omega_J$ for a fractional ideal $J$ of $K$, set 
\begin{displaymath}
\mathcal{D}_{\nu_i}=\begin{pmatrix}
  \mathcal{O}_{\nu_i} & J\mathcal{O}_{\nu_i}\\
  J^{-1}\mathcal{O}_{\nu_i} &  \mathcal{O}_{\nu_i}
\end{pmatrix},
\qquad
\mathcal{D}_{\nu_i}^\prime=
\begin{pmatrix}
  \mathcal{O}_{\nu_i} & \pi_{\nu_i}^{-1}J\mathcal{O}_{\nu_i}\\
  \pi_{\nu_i}J^{-1}\mathcal{O}_{\nu_i} &  \mathcal{O}_{\nu_i}
\end{pmatrix},
\end{displaymath}
In both cases $\mathcal{D}_{\nu_i}$ is conjugate to $\mathcal{D}_{\nu_i}^\prime$ by 
$\begin{pmatrix}
  1 & 0\\
 0 &  \pi_{\nu_i}
\end{pmatrix}$ and $\Omega\subset \mathcal{D}_{\nu_i}\cap\mathcal{D}_{\nu_i}^\prime$. As this holds for all $i$ and 
$\Omega\subset\mathcal{D}_\nu$ for all $\nu\neq\nu_i$ for any $i$ by definition of $\mathcal D$, we have, for every 
$\gamma$, $\Omega\subset \cap_{\nu}(\A\cap\mathcal{D}_\nu^\gamma)=\mathcal{D}^\gamma$. Thus $\Omega$ can be embedded 
into representatives of every isomorphism class in the genus of $\mathcal{R}$.\end{proof}

In light of Proposition \ref{proposition:case1} we henceforth assume that $\Omega$ is an integral domain with quotient field $L$.

\begin{proposition}\label{proposition:case2} 
If $L\not\subset K(\mathcal{R})$, then every order in the 
genus of $\mathcal{R}$ admits an embedding of $\Omega$.
\end{proposition}

\begin{proof}
By Lemma \ref{lemma:splitornot} we can choose 
primes $\nu_i$ of $K$ which split in $L$ such that $\{\overline{e_{\nu_i}}\}$ 
generate $G_\mathcal{R}$. Let $\lambda=\nu_i$ for some $i$. As $\lambda$ splits in 
$L/K$, $\A_\lambda$ must be split (this is immediate from the Albert-Brauer-Hasse-Noether Theorem), 
so there is a $K_\lambda-$isomorphism $f_\lambda :\A_\lambda \longrightarrow M_2(K_\lambda)$ 
such that
\begin{displaymath}
f_\lambda(L)\subset \begin{pmatrix}
  K_\lambda & 0\\
 0 &  K_\lambda
\end{pmatrix}
\end{displaymath}
by Lemma 2.2 of \cite{brztwo}. Then
\begin{displaymath}
f_\lambda(\Omega)\subset f_\lambda(\mathcal{O}_L)\subset\begin{pmatrix}
  \mathcal{O}_{K_\lambda} & 0\\
 0 &  \mathcal{O}_{K_\lambda}
\end{pmatrix}.
\end{displaymath}
Choose $\mathcal{D}_{\lambda}$ and $\mathcal{D}_{\lambda}^\prime$ so that
\begin{displaymath}
\mathcal{D}_{\lambda}=f_\lambda^{-1}(\begin{pmatrix}
  \mathcal{O}_{K_\lambda} &   \mathcal{O}_{K_\lambda}\\
   \mathcal{O}_{K_\lambda} &  \mathcal{O}_{K_\lambda}
\end{pmatrix}),\qquad\qquad \mathcal{D}_{\lambda}^\prime =f^{-1}_\lambda(\begin{pmatrix}
  \mathcal{O}_{K_\lambda} & \pi_\lambda^{-1}\mathcal{O}_{K_\lambda}\\
  \pi_\lambda\mathcal{O}_{K_\lambda} &  \mathcal{O}_{K_\lambda}
\end{pmatrix}).
\end{displaymath}
As $\lambda$ ranges over all $\nu_i$, we see that $\Omega\subset\mathcal{D}_{\nu_i}\cap\mathcal{D}_{\nu_i}^\prime$ 
for all $i$. By the strategy outlined in the beginning of this section, $\Omega$ can be embedded into representatives of every isomorphism 
class in the genus of $\mathcal{R}$.
\end{proof}

Recalling that the conductor of $K(\mathcal R)$ is divisible only by the prime divisors of the level ideal $N_\mathcal R$ of $\mathcal R$, we have the following immediate corollary to Proposition \ref{proposition:case2}.

\begin{corollary}\label{corollary:notdividinglevelideal} If any finite prime of $K$ not dividing $N_\mathcal R$ ramifies in $L$, then every order in the genus of $\mathcal{R}$ admits an embedding of $\Omega$.\end{corollary}

That Corollary \ref{corollary:notdividinglevelideal} allows one to preclude the possibility of selectivity without computing the classfield $K(\mathcal R)/K$ is especially nice, as in practice such a computation may be quite difficult.

Having dealt with the case that $L\not\subset K(\mathcal R)$, we now suppose that $L\subset K(\mathcal R)$.

\begin{proposition}\label{proposition:Linclassfield} If $L\subset K(\mathcal{R})$, then the proportion of isomorphism classes in the genus of $\mathcal R$ whose representatives admit an embedding of $\Omega$ is at least $1/2$.
\end{proposition}

\begin{proof}
Identify $\Omega$ with its image in $\mathcal R$ so that $\Omega\subset \mathcal R$. We may choose, by Lemma \ref{lemma:splitornot}, a generating set $\{\overline{e_{\nu_i}}\}_{i=1}^m$ of $G_\mathcal{R}$ for which $\nu_1$ is inert in $L/K$, and $\nu_i$ splits in $L/K$ whenever $i>1$. Let $\{\mathcal{D}^\gamma\}$ be a parameterization of the genus of $\mathcal{R}$ associated to this generating set. The proof of Proposition \ref{proposition:case2} shows that the parameterization $\{\mathcal D^\gamma\}$ can be constructed so that for all $\gamma$, $\Omega$ is contained in $\mathcal{D}^\gamma_{\nu_i}$ whenever $i>1$. As the set of orders $\{ \mathcal D^\gamma :  \mathcal D^\gamma_{\nu_1}=\mathcal R_{\nu_1} \}$ represent half of the isomorphism classes in the genus of $\mathcal R$ and each of these orders contains $\Omega$, we're done. \end{proof}

We summarize Propositions \ref{proposition:case2} and \ref{proposition:Linclassfield} as a theorem.

\begin{theorem}\label{theorem:embedseverywhere} Assume that an embedding of $\Omega$ into $\mathcal R$ exists.
	\begin{enumerate}
		\item If $L\not\subset K(\mathcal{R})$, then every order in the 
		genus of $\mathcal{R}$ admits an embedding of $\Omega$.
		
		\item If $L\subset K(\mathcal{R})$, then the proportion of isomorphism classes in the genus of $\mathcal R$ whose representatives admit an embedding of $\Omega$ is at least $1/2$.
	\end{enumerate}

\end{theorem}

\subsection{A selectivity theorem}\label{subsection:aselectivitytheorem}

In this section we constrain slightly the class of orders $\mathcal R \subset \A$ that we consider and provide necessary and sufficient conditions for an order $\Omega$ to be selective for the genus of $\mathcal R$. In the case that $\Omega$ is selective for the genus of $\mathcal R$, we shall see that representatives of exactly $1/2$ of the isomorphism classes of the genus of $\mathcal R$ admit an embedding of $\Omega$. The constraints which we shall impose on $\mathcal R$ are the following:
\begin{equation}\label{first}
	\mbox{The relative discriminant ideal } d_{\Omega/\mathcal{O}_K} \mbox{ of } \Omega \mbox{ is coprime to the level ideal } N_\mathcal R \mbox{ of } \mathcal R. 
\end{equation}
\begin{equation}\label{second}
	\mbox{The set of primes dividing } N_\mathcal R  \mbox{ is disjoint from the set of primes ramifying in } \A. 
\end{equation}

\begin{theorem}\label{theorem:main}
Assume that an embedding of $\Omega$ into $\mathcal{R}$ exists. Then every order in the genus of $\mathcal{R}$ admits an embedding of $\Omega$ except when the following conditions hold:

\begin{enumerate}

\item $\Omega$ is an integral domain whose quotient field $L$ is a quadratic field extension of $K$ which is contained in $\A$.

\item There is a containment of fields $L\subset K(\mathcal{R})$.

\item All primes of $K$ which divide the relative discriminant ideal $d_{\Omega/\mathcal{O}_K}$ of $\Omega$ split in $L/K$.

\end{enumerate}

Suppose now that (1) - (3) hold. Then the isomorphism classes in the genus of $\mathcal{R}$ whose representatives admit an embedding of $\Omega$ comprise exactly half of the isomorphism classes. If $\mathcal{R}$ admits an embedding of $\Omega$ and $\mathcal E$ is another order in the genus of $\mathcal R$, then $\mathcal E$ admits an embedding of $\Omega$ if and only if $\displaystyle\FrobclassfieldL(\rho(\mathcal R,\mathcal E))$ is the trivial element of $Gal(L/K)$.
\end{theorem}

\begin{remark} Although our condition (2) is not the same as that which appears in \cite{chan-xu, chinburg, guo-qin}, it is equivalent whenever
$\mathcal{R}$ is an Eichler order. Indeed, this equivalence plays an important role in the proofs of the aforementioned theorems.\end{remark}

\begin{example} Let $K$ be a number field and $\A$ be a quaternion algebra over $K$ which satisfies the Eichler condition and has no finite ramification. Let $L$ be a maximal subfield of $\A$ and $\Omega=\mathcal O_L$. If $\mathcal R$ is any quaternion order in $\A$ containing $\Omega$ for which assumption (\ref{first}) holds, then by Theorem \ref{theorem:main}, the proportion of isomorphism classes in the genus of $\mathcal R$ whose representatives admit an embedding of $\Omega$ is $[K(\mathcal R)\cap L: K]^{-1}$. This generalizes the $n=2$ case of a theorem of Chevalley \cite{Chevalley-book}, who showed that if $L$ is a maximal subfield of $M_n(K)$ then the ratio of the isomorphism classes of maximal orders of $M_n(K)$ into which $\mathcal O_L$ can be embedded to the total number of isomorphism classes of maximal orders is $[H_K\cap L : K]^{-1}$, where $H_K$ is the Hilbert class field of $K$. 
	
\end{example}

We prove Theorem \ref{theorem:main} as a series of propositions. We have seen in Propositions \ref{proposition:case1} and \ref{proposition:case2} that if either (1) or (2) fail to hold then every order in the genus of $\mathcal R$ admits an embedding of $\mathcal R$. We henceforth assume that conditions (1) and (2) hold.

\begin{proposition}\label{propositon:1and2hold} Both $L/K$ and $\A$ are unramified at all finite primes of $K$ and ramify at the same (possibly empty) set of real primes.\end{proposition}
\begin{proof}
If a finite prime $\nu$ of $K$ ramifies in $\A$ then by assumption \eqref{second}, $\mathcal R_\nu$ must be maximal, hence $n(N(\mathcal R_\nu))=K_\nu^\times$. It follows that $\nu$ splits completely in $K(\mathcal R)$ and thus in $L/K$ as well. As $L$ embeds into $\A$ we obtain a contradiction to the Albert-Brauer-Hasse-Noether theorem. If $\nu$ is a finite prime of $K$ which ramifies in $L/K$ then $\nu$ ramifies in $K(\mathcal R)/K$ and thus divides the level ideal $N_\mathcal R$. As $\nu$ also divides $d_{\Omega/\mathcal{O}_K}$ we have a contradiction to assumption \eqref{first}. That $L/K$ is ramified at every real prime which ramifies in $\A$ follows from the Albert-Brauer-Hasse-Noether theorem. If $\nu$ is a real prime of $K$ which is unramified in $\A$ then $n(N(\mathcal R_\nu))=\mathbb R^\times$, hence $\nu$ splits completely in $K(\mathcal R)/K$ and thus in $L/K$ as well.\end{proof}

\begin{remark} That $L/K$ and $\A$ are unramified at all finite primes and ramify at exactly the same set of real primes appeared as condition (2) in \cite{chinburg, guo-qin} and may very well be a necessary condition for selectivity to occur at all in $\A$. \end{remark}
	
\begin{proposition}\label{proposition:case3} 
If condition (3) does not hold, then every order in the genus of $\mathcal{R}$ admits an embedding of $\Omega$.
\end{proposition}

\begin{proof}
	
If condition (3) does not hold, there exists a finite prime $\lambda$ of $K$ with $\lambda\mid d_{\Omega/{\mathcal{O}_K}} $ and which does not split in $L/K$. By Proposition \ref{propositon:1and2hold}, $\lambda$ is unramified in $L/K$, so $\lambda$ is inert in $L$. By assumption \eqref{first}, $(d_{\Omega/{\mathcal{O}_K}},N_\mathcal R)=1$ and since $\lambda \mid d_{\Omega/{\mathcal{O}_K}}$, we have $\lambda\nmid N_\mathcal R$ which implies that $\mathcal R_\lambda$ is maximal. Finally, since $\lambda$ is finite, Proposition \ref{propositon:1and2hold} implies that $\A_\lambda\cong M_2(K_\lambda)$.

Because condition (2) holds, $L\subset K(\mathcal R)$ and by Lemma \ref{lemma:splitornot}
we may choose $\nu_2,...,\nu_m$ such that $\{\overline{e_{\lambda}},\overline{e_{\nu_2}},..,\overline{e_{\nu_m}} \}$ generate $G_\mathcal{R}$, where $\nu_i$ splits in $L/K$ 
for $i=2,..,m$. We also assume that the $\nu_i$ are chosen so that $\mathcal R_{\nu_i}$ is maximal for all $i$. As the $\nu_i$  all split in $L$, for each $i$ we may pick two adjacent maximal orders of $\A_{\nu_i}$ containing $\Omega$ (as in the proof of Proposition \ref{proposition:case2}).

Since $L_\lambda / K_\lambda$ is a quadratic unramified extension and $\lambda$ divides $d_{\Omega/\mathcal{O}_K}$,
$\Omega\subset \Omega\otimes_{\mathcal{O}_K}\mathcal{O}_{K_\lambda}\subset\mathcal{O}_{K_\lambda}+\pi_\lambda\mathcal{O}_{L_\lambda}$. 
Let $\mathcal{D}_\lambda$ be some maximal order of $\A_\lambda$ containing $\mathcal{O}_{L_\lambda}$ and consequently, $\Omega$. All maximal 
orders of $\A_\lambda$ are conjugate, so we may assume that $\mathcal{D}_\lambda=\begin{pmatrix}
  \mathcal{O}_{K_\lambda} & \mathcal{O}_{K_\lambda}\\
 \mathcal{O}_{K_\lambda} &  \mathcal{O}_{K_\lambda}
\end{pmatrix}$. Now consider the maximal order $\mathcal{D}_\lambda^\prime=\begin{pmatrix}
  \mathcal{O}_{K_\lambda} & \pi_\lambda^{-1}\mathcal{O}_{K_\lambda}\\
 \pi_\lambda\mathcal{O}_{K_\lambda} &  \mathcal{O}_{K_\lambda}
\end{pmatrix}$. It is clear that $\pi_\lambda\mathcal{D}_\lambda\subset\mathcal{D}_\lambda^\prime$ 
(viewed as $\mathcal{O}_{K_\lambda}$-modules), so $\pi_\lambda\mathcal{O}_{L_\lambda}\subset\mathcal{D}_\lambda^\prime$. Then 
$\Omega\subset\mathcal{D}^\prime_\lambda$.

Hence $\Omega\subset\mathcal{D}_\nu\cap\mathcal{D}_\nu^\prime$ for all $\nu\in\{\lambda,\nu_2,..,\nu_m\}$ and thus can be embedded into representatives of every isomorphism class in the genus of $\mathcal{R}$.\end{proof}

We now assume that conditions (1) - (3) hold.

\begin{proposition}\label{proposition:case4}
Let notation be as above and suppose that conditions (1) - (3) hold. Then the orders in the genus of $\mathcal{R}$ 
admitting an embedding of $\Omega$ represent exactly half of the isomorphism classes. If $\mathcal E$ is another 
order in the genus of $\mathcal R$ , then $\mathcal E$ admits an embedding of $\Omega$ if and only if $\displaystyle\FrobclassfieldL(\rho(\mathcal R,\mathcal E))$ is the trivial element of $Gal(L/K)$.\end{proposition}

\begin{proof} By assumption, $\Omega$ is contained in $\mathcal{R}$. Let $\mathcal{E}$ be an order 
in the genus of $\mathcal{R}$. We may choose, by Lemma \ref{lemma:splitornot},  a 
generating set $\{\overline{e_{\nu_i}}\}_{i=1}^m$ of $G_\mathcal{R}$ for which $\nu_1$ is inert 
in $L/K$, and $\nu_i$ splits in $L/K$ whenever $i>1$. Let $\{\mathcal{D}^\gamma\}$ be the parameterization
of the genus of $\mathcal{R}$ associated to this generating set.  As 
the $\{\mathcal{D}^\gamma\}$ span the isomorphism classes in the genus of $\mathcal{R}$, there
exists a $\gamma$ such that $\mathcal{E}\cong\mathcal{D}^\gamma$. Then $\mathcal{E}$ contains a conjugate of 
$\Omega$ if and only if $\mathcal{D}^\gamma$ contains a conjugate of $\Omega$, so we may suppose without loss of generality that $\mathcal{E}=\mathcal{D}^\gamma$.

The proof of Proposition \ref{proposition:case2} shows that the parameterization $\{\mathcal D^\gamma\}$ can be constructed so that for all $\gamma$, $\Omega$ is contained in $\mathcal{D}^\gamma_{\nu_i}$ whenever $i>1$. Suppose that $\displaystyle\Frob(\rho(\mathcal{R},\mathcal{D}^\gamma))$ is trivial in $Gal(L/K)$. Since $\rho(\mathcal R,\mathcal D^\gamma)=\prod_{i=1}^m \overline{e_{\nu_i}}^{\gamma_i}$, this implies that $\displaystyle\Frob(\overline{e_{\nu_1}}^{\gamma_1})$ is trivial (as $\nu_i$ 
splits in $L/K$ whenever $i>1$). As $\displaystyle\Frob(\overline{e_{\nu_1}})$ is 
the non-trivial element of $\displaystyle\Gal(L/K)$, we deduce that $\gamma_1=0$. Therefore 
$\Omega\subset \mathcal{D}^\gamma_{\nu_1}$ and $\Omega\subset\mathcal{D}^\gamma_\nu$ for 
all primes $\nu$ of $K$, hence $\Omega\subset\mathcal{D}^\gamma$.

In order to prove the converse we will need a lemma.

\begin{lemma}\label{lemma:uniquevertex}
If $\nu$ is a finite prime of $K$ which is inert in $L$, then $\mathcal O_{L_\nu}$ is contained in a unique maximal order of $\A_\nu$.	
\end{lemma}
\begin{proof}
By Proposition \ref{propositon:1and2hold}, $\nu$ is unramified in $\A$ so that we may identify $\A_\nu$ with $M_2(K_\nu)$. Suppose that $\mathcal O_{L_\nu}$ is contained in distinct maximal orders $\mathcal M_1, \mathcal M_2$ of $\A_\nu$. All of the maximal orders of $M_2(K_\nu)$ are conjugate, so there exists an element $x\in\A_\nu^\times$ such that $\mathcal M_2 = x\mathcal M_1 x^{-1}$. Because $\nu$ is inert in $L/K$ we may write $\mathcal O_{L_\nu}=\mathcal O_{K_\nu}[\alpha]$ for some $\alpha\in\mathcal O_{L_\nu}^\times$. Because $\alpha\in\mathcal M_1^\times\ \cap \mathcal M_2^\times$, conjugation by $\alpha$ fixes both $\mathcal M_1$ and $\mathcal M_2$ and hence every vertex in the unique path joining $\mathcal M_1$ and $\mathcal M_2$ in the tree of maximal orders of $\A_\nu$. Thus $\alpha$ fixes an edge in the tree of maximal orders of $\A_\nu$. But this contradicts Lemma 2.2 of \cite{chinburg}, proving the lemma.
\end{proof}

We now continue the proof of Proposition \ref{proposition:case4} and suppose now that $\mathcal{E}$ contains a conjugate of $\Omega$ but $\displaystyle\Frob(\rho(\mathcal{R},\mathcal{E}))$ is not trivial. By conjugating $\mathcal{E}$ we may assume that $\Omega$ is contained in $\mathcal{E}$, as $\displaystyle\Frob(\rho(\mathcal{R},\mathcal{E}))$ is unchanged if $\mathcal{E}$ is replaced with a conjugate order (by Proposition \ref{proposition:distanceproperties}). We claim that there is a finite prime $\nu$ which is inert in $L/K$ and for which $\mathcal{R}_\nu$ and $\mathcal{E}_\nu$ are not equal. Suppose to the contrary that $\mathcal E_\mu=\mathcal R_\mu$ for all primes $\mu$ that are inert in $L/K$. Let $y=\tilde{y}_\mathcal E\in J_\A$ be such that $\mathcal E_\lambda=y_\lambda  \mathcal R_\lambda  y_\lambda^{-1}$ for all primes $\lambda$ of $K$. By hypothesis we may take $y_{\mu}=1$ for all primes $\mu$ which are inert in $L/K$. We may also take $y_\mu=1$ for all archimedean primes $\mu$ of $K$.  Consider the element $\rho(\mathcal R,\mathcal E)=n(y)H_\mathcal R$ of $G_\mathcal R$. Its image under $\Frob$ is clearly trivial since $n(y_\mu)=1$ whenever $\mu$ is inert in $L/K$ (here we have used the fact that by Proposition \ref{propositon:1and2hold}, $L/K$ is unramified at all finite primes). This contradicts our assumption that $\displaystyle\Frob(\rho(\mathcal{R},\mathcal{E}))$ be non-trivial and proves our claim.

Let $\nu$ be a finite prime which is inert in $L/K$ and such that $\mathcal R_\nu\ne \mathcal E_\nu$. By Proposition \ref{propositon:1and2hold} $\nu$ is unramified in $\A$. Because condition (3) holds, $\Omega_\nu=\mathcal O_{L_\nu}$. Locally we see that $\Omega_\nu\subset\mathcal R_\nu \subset \mathcal M_\nu$ for some maximal order $\mathcal M_\nu$ of $\A_\nu$. Writing $\mathcal E_\nu = y_\nu\mathcal R_\nu y_\nu^{-1}$, we have $\Omega_\nu\subset \mathcal E_\nu \subset y_\nu \mathcal M_\nu y_\nu^{-1}$. By Lemma \ref{lemma:uniquevertex}, $\mathcal M_\nu=y_\nu\mathcal M_\nu y_\nu^{-1}$, so that $\mathcal M_\nu$ is fixed by $y_\nu$. An element of $\A_\nu$ whose reduced norm has odd valuation fixes no maximal order, so it must be the case that $\ordone(n(y_\nu))$ is even. If this holds for every inert prime $\nu$ for which $\mathcal R_\nu\ne \mathcal E_\nu$, then  $\displaystyle\Frob(\rho(\mathcal{R},\mathcal{E}))$ would be trivial. This contradiction finishes the proof. \end{proof}


\section{An Optimal Embedding Theorem}\label{section:optimal}
Let $\Omega$ be an integral domain whose quotient field $L$ is a quadratic extension of $K$. Recall that an embedding  $\sigma: L\longrightarrow \A$ is an optimal embedding of $\Omega$ into $\mathcal{R}$ if $\sigma(L)\cap \Rc=\sigma(\Omega).$ Given a finite prime $\nu$ of $K$ and an embedding $\sigma: L\longrightarrow \A$, extension of scalars yields an embedding $\sigma_\nu: L_\nu\longrightarrow \A_\nu.$ The following lemma makes clear the relationship between the optimality of an embedding $\sigma$ and the optimality of the induced embeddings $\{\sigma_\nu\}$, where
$\nu$ ranges over the finite primes of $K$.

\begin{lemma}\label{lemma:locglob} Let notation be as above. Then $\sigma$ is an optimal embedding of $\Omega$ into $\Rc$ if and only if for every finite prime $\nu$ of $K$, $\sigma_\nu$ is an optimal embedding
of $\Omega_\nu$ into $\Rc_\nu$.
\end{lemma}
\begin{proof} Identify $L$ and $\Omega$ with their images in $\A$ under $\sigma$. We must show that $L\cap\Rc=\Omega\mbox{ if and only if } L_\nu\cap\Rc_\nu=\Omega_\nu \mbox { for all } \nu<\infty.$ This follows from the Corollary on page 85 of \cite{weil}, which implies that $(L\cap\Rc)_\nu=L_\nu\cap\Rc_\nu$ for all finite primes $\nu$ of $K$.\end{proof}

All of the results in Section \ref{section:main} hold for optimal embeddings as well. We begin by proving a few general results in which there are no restrictions on the orders $\Omega$ or $\mathcal R$. Assume there is an optimal embedding $\sigma$ of $\Omega$ into $\Rc$.

\begin{proposition}\label{proposition:optimalcase2} 
If $L\not\subset K(\mathcal{R})$, then every order in the 
genus of $\mathcal{R}$ admits an optimal embedding of $\Omega$.
\end{proposition}
\begin{proof}
If $L\not\subset K(\mathcal{R})$, we may choose a generating set $\{\overline{e}_{\nu_i}\}_{i=1}^m$ of $G_\Rc$ 
	such that both $\Rc_{\nu_i}$ and $\Omega_{\nu_i}$ are maximal for all $i$. As in the proof of 
	Proposition \ref{proposition:case2}, we may construct a parameterization $\{\mathcal{D}^\gamma\}$ of 
	the genus of $\Rc$ such that each $\mathcal{D}^\gamma_{\nu_i}$ contains $\sigma_{\nu_i}(\Omega_{\nu_i})$. 
	However every embedding of a maximal quadratic order into a quaternion order is optimal. In particular, $\sigma_{\nu_i}$ is 
	an optimal embedding of $\Omega_{\nu_i}$ into $\mathcal{D}^\gamma_{\nu_i}$ for all $\gamma$ and for $i=1,\dots,m$. 
	As $\sigma$ is an optimal embedding of $\Omega$ into $\mathcal{R}$, Lemma \ref{lemma:locglob} implies that $\sigma_\nu$ is
	an optimal embedding of $\Omega_\nu$ into $\mathcal{R}_\nu$ for all $\nu$. As $\mathcal{D}^\gamma_\nu=\mathcal{R}_\nu$  
	whenever $\nu\not\in\{\nu_i\}$, we conclude from Lemma \ref{lemma:locglob} that $\sigma$ is an optimal embedding of $\Omega$ into $\mathcal{D}^\gamma$.
\end{proof}

\begin{proposition}\label{proposition:optimalLinclassfield} If $L\subset K(\mathcal{R})$, then the proportion of isomorphism classes in the genus of $\mathcal R$ whose representatives admit an optimal embedding of $\Omega$ is at least $1/2$.
\end{proposition}
\begin{proof}
Let $\{\mathcal{D}^\gamma\}$ be a parameterization of the genus of $\Rc$ and let $\mathcal{D}^{\gamma_0}\in\{\mathcal{D}^\gamma\}$ contain $\sigma(\Omega)$ (i.e. $\sigma$ is an embedding of $\Omega$ into $\mathcal{D}^{\gamma_0}$). We show that for every finite prime $\nu$ of $K$, $\sigma_\nu$ is an optimal embedding of $\Omega_\nu$ into
	$\mathcal{D}^{\gamma_0}_\nu$. If $\nu\not\in\{\nu_i\}$ then $\mathcal{D}^{\gamma_0}_\nu=\Rc_\nu$ and $\sigma_\nu$ is optimal by Lemma \ref{lemma:locglob}. As $\Omega_{\nu_i}$ can be taken to be the maximal order of $L_{\nu_i}$ (by suitably choosing the generators $\{\overline{e}_{\nu_i}\}$ of $G_\Rc$) and every embedding of a maximal quadratic order into a quaternion order is optimal, we see that 
	the embeddings $\{\sigma_{\nu_i}\}_{i=1}^m$ are optimal as well. By Lemma \ref{lemma:locglob} $\sigma$ is an optimal embedding of $\Omega$ into $\mathcal{D}^{\gamma_0}$, finishing our proof.
\end{proof}

We now adopt the assumptions outlined in the beginning Section \ref{subsection:aselectivitytheorem} and prove an analog of Theorem \ref{theorem:main} for optimal embeddings.

\begin{theorem}\label{theorem:optimalmain}
Let notation be as above and suppose that assumptions (\ref{first}) and (\ref{second}) are satisfied. If an optimal embedding of $\Omega$ into $\mathcal{R}$ exists then every order in the genus of $\mathcal{R}$ admits an optimal embedding of $\Omega$ except when the following conditions hold:

\begin{enumerate}

\item We have a containment of fields $L\subset K(\mathcal{R})$.

\item All primes of $K$ which divide the relative discriminant ideal $d_{\Omega/\mathcal{O}_K}$ of $\Omega$ split in $L/K$.

\end{enumerate}

Suppose now that (1) and (2) hold. Then the isomorphism classes in the genus of $\mathcal{R}$ whose representatives admit an optimal embedding of $\Omega$ comprise exactly half of the isomorphism classes. If $\mathcal{R}$ admits an optimal embedding of $\Omega$ and $\mathcal E$ is another order in the genus of $\mathcal R$, then $\mathcal E$ admits an optimal embedding of $\Omega$ if and only if $\displaystyle\FrobclassfieldL(\rho(\mathcal R,\mathcal E))$ is the trivial element of $Gal(L/K)$.
\end{theorem}

\begin{proof}

If condition (1) fails then Proposition \ref{proposition:optimalcase2} shows that every order in the genus of $\mathcal R$ admits an optimal embedding of $\Omega$.

If condition (1) is satisfied and condition (2) fails, then we may choose a generating set $\{\overline{e}_{\nu_i}\}_{i=1}^m$ of $G_\Rc$ 
such that both $\Rc_{\nu_i}$ and $\Omega_{\nu_i}$ are maximal for all $i$. The proof of Theorem \ref{theorem:main} again proceeds by
constructing a parameterization $\{\mathcal{D}^\gamma\}$ of the genus of $\Rc$ such that for every $i$, $\sigma(\Omega_{\nu_i})\subset \mathcal{R}_{\nu_i}\cap\mathcal{R}^\prime_{\nu_i}$, where $\mathcal{D}^\gamma_{\nu_i}$ equals either $\mathcal{R}_{\nu_i}$ or $\mathcal{R}^\prime_{\nu_i}$, depending on the parity of $\gamma_i$. The argument used to prove Proposition \ref{proposition:optimalcase2} shows that $\sigma$ is an optimal embedding of $\Omega$ into every order in the genus of $\Rc$. 

We now assume that (1) and (2) are both satisfied and show that the orders in the genus of $\mathcal{R}$
admitting an optimal embedding of $\Omega$ represent exactly one half of the isomorphism classes. Proposition \ref{proposition:optimalLinclassfield} shows that the representatives of at least half of the isomorphism classes of orders in the genus of $\mathcal R$ admit an optimal embedding of $\Omega$. By Theorem \ref{theorem:main}, the orders $\E$ in the genus of $\Rc$ admitting an embedding of $\Omega$ are precisely those for which $\displaystyle\FrobclassfieldL(\rho(\mathcal R,\mathcal E))$ is the trivial element of $Gal(L/K)$ and represent exactly half of the isomorphism classes. An order in the genus of $\mathcal{R}$ cannot admit an optimal embedding of $\Omega$ if it admits no embedding of $\Omega$, so the isomorphism classes whose representatives admit an optimal embedding of $\Omega$ are precisely those whose representatives admit an embedding of $\Omega$, finishing our proof. \end{proof}


\end{document}